\documentclass{amsart}

\usepackage{amssymb}

\usepackage[dvipsnames]{xcolor}

\DeclareMathOperator{\ot}{ot}

\DeclareMathOperator{\cf}{cf}

\DeclareMathOperator{\pcf}{pcf}

\DeclareMathOperator{\tcf}{tcf}

\DeclareMathOperator{\lh}{lh} 

\newcommand{\lex}{<_{\rm lex}}

\usepackage{todonotes}

\setuptodonotes{inline,color=teal!20}

\newtheorem{theorem}{Theorem}

\newtheorem{lemma}{Lemma}

\newtheorem{remark}{Remark}

\newtheorem{definition}{Definition}

\newtheorem*{fact}{Fact}

\title{Squares, scales and lines} 

\author{James Cummings} \thanks{The author was partially supported by NSF grant DMS–2054532.}

\begin{document}

\begin{abstract}
  We use hypotheses from PCF theory to construct a linear ordering
  which has cardinality the successor of a singular cardinal of countable cofinality, and is
  {\em incompact} in the following sense: the ordering is not sigma-scattered,
  but every smaller subordering is sigma-wellordered. Such orderings were
  first constructed by Todorcevic \cite[Section 7.6]{Stevobook} using Jensen's square principle. 
\end{abstract}

\maketitle

\section{Introduction}

Todorcevic \cite[Section 7.6]{Stevobook} proved that
if $\kappa$ is a singular
cardinal of cofinality $\omega$ and $\square_\kappa$ holds, then there is a
linear ordering of cardinality $\kappa^+$ which has density $\kappa$ (in particular
it is not $\sigma$-scattered) while every subordering of size at most $\kappa$
is $\sigma$-wellordered.\footnote{See Section \ref{LOs} for some background
about scattered and $\sigma$-scattered linear orderings.}
This is a rather typical application of $\square_\kappa$ to
construct a structure of size $\kappa^+$ which is {\em incompact}, in the sense 
that it is different in some way from all of its smaller substructures. 

The argument by Todorcevic goes through the theory of minimal walks and $\rho$-functions.
Starting with a suitable $\square_\kappa$-sequence, the minimal walk on that sequence is used to
construct a function $\rho:[\kappa^+]^2 \rightarrow \kappa$ with strong combinatorial properties:
roughly speaking the function $\rho$ ``remembers'' that the club sets in the $\square_\kappa$-sequence
are coherent and have small order type. The function $\rho$ is used to construct a sequence of functions which
is (at least morally speaking) a {\em scale} in the sense of PCF theory, and then the desired linear ordering
is read off from the scale. 

A central idea in the author's prior work with Foreman and Magidor on singular cardinal combinatorics
(see especially \cite{CFM}) is that principles from PCF theory can often do the work of
$\square_\kappa$ and its relatives. This turns out to be the case here: PCF-theoretic principles  
studied in \cite{CFM} can do the work of $\square_\kappa$ in constructing incompact linear orderings
of the type discussed above, and this permits us to put some other hypotheses in place of $\square_\kappa$. 

Our main results are that:
\begin{itemize}
\item If $\kappa$ is a singular
cardinal of cofinality $\omega$ and the weak square principle $\square^*_\kappa$ holds, then there is a
linear ordering of cardinality $\kappa^+$ which has density $\kappa$ while every subordering of size at most $\kappa$
is $\sigma$-wellordered.
\item If $\kappa$ is singular strong limit of cofinality $\omega$ and $2^\kappa > \kappa^+$, then
  the same conclusion holds.
\end{itemize}

In fact we will derive the conclusion from a PCF hypothesis (the existence of a {\em disjointing scale}) and
derive such a scale from both the given hypotheses. We note for the record that neither of our hypotheses
implies the other: if $V = L$ then $\square_\kappa$ holds and $2^\kappa = \kappa^+$ for all $\kappa$,
while work of Gitik and Sharon \cite{GitikSharon} gives models where $2^\kappa > \kappa^+$ and $\square^*_\kappa$ fails
for some singular strong limit $\kappa$ of cofinality $\omega$. 

In a short appendix we spell out for the interested reader the relation between 
the constructions from \cite[Section 7.6]{Stevobook} and PCF-theoretic scales.

\section{Preliminaries}

\subsection{Linear orderings} \label{LOs}

A linear ordering $L$ is {\em scattered} if and only if  $(\mathbb Q, <)$ does not embed into $L$.
Hausdorff \cite{Hausdorff} proved that the class of scattered orderings is the least class which contains
the one point ordering and is closed under wellordered and reverse wellordered lexicographic sums.
It follows that the property of being scattered is absolute, and each scattered ordering can
be ``certified'' by a description of how it may be built.

Using this description of scattered sets an easy induction shows:
\begin{fact}(Hausdorff) If $\lambda$ is an infinite regular cardinal and
  $L$ is a scattered set with $\vert L \vert \ge \lambda$, then
  at least one of $\lambda$ and $\lambda^*$ (the reverse of $\lambda$) embeds into $L$.
\end{fact}

A linear ordering $L$ is {\em $\sigma$-scattered} if and only if it is the union
of countably many scattered subsets, and {\em $\sigma$-wellordered} 
if and only if it is the union of countably many wellordered subsets.
It is easy to see that:
\begin{fact} If $\lambda$ is an  regular uncountable cardinal and
  $L$ is a $\sigma$-scattered set with $\vert L \vert \ge \lambda$, then
  at least one of $\lambda$ and $\lambda^*$ (the reverse of $\lambda$) embeds into $L$.
\end{fact}   

Given a linear ordering $L$ we will say that $D \subseteq L$ is {\em weakly dense}
if and only if for all $a, c \in L$ with $a < c$ there is $b \in D$ with $a \le b \le c$.
We will say that $L$ {\em has density $\kappa$} if it has a weakly dense subset of cardinality
$\kappa$. 
It is easy to see that if $\kappa$ is an infinite cardinal and $L$ has
density $\kappa$ then neither $\kappa^+$ nor its reverse can embed into $L$.
Combining this with the preceding fact:
\begin{fact} 
  If $\kappa$ is an infinite cardinal,
  $\vert L \vert \ge \kappa^+$ and $L$ has density $\kappa$, then $L$ is not $\sigma$-scattered.
\end{fact}

\subsection{Scales}

We summarise some information about PCF theory, with a focus on singular cardinals
of cofinality $\omega$ and reduced products taken modulo the ideal of finite sets.

A classical result by Shelah \cite{CardArith} implies that if
$\kappa$ is a singular cardinal of cofinality $\omega$ then there exists
an increasing sequence of regular cardinals $(\kappa_n)_{n < \omega}$ cofinal in
$\kappa$, together with a sequence $(f_\alpha)_{\alpha < \kappa^+}$
which is increasing and cofinal in $\prod_n \kappa_n$ ordered by eventual domination. 
We say that $(f_\alpha)_{\alpha < \kappa^+}$ is a {\em scale of length $\kappa^+$ in $\prod_n \kappa_n$}:
 in the obvious way we also define scales of length $\mu$ for any regular $\mu > \kappa$.

For our purposes it will be convenient to have a slightly more generous notion.
Let $(\lambda_m)_{m < \omega}$ be a sequence of ordinals such that $\lambda_m < \kappa$
and $\cf(\lambda_m) \rightarrow \kappa$, that is to say that
$\{ m : \cf(\lambda_m) < \eta \}$ is finite for  all  $\eta < \kappa$.
Let $<^*$ be the relation of eventual domination. 
We will say that a {\em weak scale of length $\mu$ in $\prod_m \lambda_m$} is
a sequence of functions $(g_\alpha)_{\alpha < \mu}$
such that:
\begin{enumerate}
\item $g_\alpha: \omega \rightarrow \kappa$.
\item $(g_\alpha)$ is $<^*$-increasing.
\item $g_\alpha(m) < \lambda_m$ for all but finitely many $m$.
\item For all $h \in \prod_m \lambda_m$, there is
  $\alpha$ such that $h <^* g_\alpha$.
\end{enumerate}
It is a standard fact that existence of a weak scale of length $\mu$ implies the existence of a scale of length $\mu$,
but we will not need this. 

%
%

We need some definitions from \cite{CFM}. In that paper they were made in the setting of scales,
but they apply equally well to weak scales. 
\begin{definition}
  Let $(f_\alpha)_{\alpha < \mu}$ be a weak scale of length $\mu$ in $\prod_n \lambda_n$,
  and let $\alpha < \mu$ be an ordinal with $\omega < \cf(\alpha) < \kappa$. 
  \begin{itemize}
  \item
    $\alpha$ is {\em good} if and only if there exist  $A \subseteq \alpha$ unbounded 
    and $m < \omega$ such that $(f_\alpha(n))_{\alpha \in A}$ is strictly increasing for all $n \ge m$.
  \item $\alpha$ is {\em better} if and only if there exists  $C \subseteq \alpha$ closed unbounded
    such that for all $\beta \in \lim(C)$ there is $m$ such that $f_\alpha(n) < f_\beta(n)$ for all
    $\alpha \in C \cap \beta$ and all $n \ge m$.
  \item   
    $\alpha$ is {\em very good} if and only if there exist  $C \subseteq \alpha$ closed unbounded
    and $m < \omega$ such that $(f_\alpha(n))_{\alpha \in C}$ is strictly increasing for all $n \ge m$.
    \end{itemize}
\end{definition} 
  It is easy to see that very good points are better, and better points are good.

  It is shown in \cite{CFM} that if $\kappa$ is singular of cofinality $\omega$ and
  $\square^*_\kappa$ holds then there is a scale of length $\kappa^+$ in which all points of uncountable cofinality are better.
  Since the proof is quite short and we need the same ideas later, we sketch the proof here.
  Let $({\mathcal C}_\alpha)_{\alpha < \kappa^+}$ be a $\square^*_\kappa$-sequence, that is:
  \begin{itemize}
  \item ${\mathcal C}_\alpha$ is a family of closed unbounded subsets of $\alpha$, each of order type
    less than $\kappa$, with $1 \le \vert {\mathcal C}_\alpha \vert \le \kappa$.
  \item For all $C \in {\mathcal C}_\alpha$ and all $\beta \in \lim(C)$,
    $C \cap \beta \in {\mathcal C}_\beta$.
  \end{itemize}  
  Let 
  $(f_\alpha)_{\alpha < \kappa^+}$ be some scale of length $\kappa^+$ in some product $\prod_n \kappa_n$.
  We will constuct an increasing subsequence $(g_\beta)_{\beta < \kappa^+}$ of  
  the scale $(f_\alpha)_{\alpha < \kappa^+}$, such that all
  points of uncountable cofinality are better in $(g_\beta)_{\beta < \kappa^+}$.

  Let $g_0 = f_0$ and let $g_{\beta+1} = f_\alpha$ where $\alpha$ is chosen minimal such that 
  $g_\beta <^* f_\alpha$.
  Suppose that we defined $(g_\beta)_{\beta < \gamma}$ for some limit ordinal $\gamma$. For each $C \in {\mathcal C}_\gamma$
  let $m$ be such that $\ot(C) < \kappa_m$, and define $h^C_\gamma(n) = \sup (\{ g_\beta(n) : \beta \in C \} )$ for $n \ge m$.
  Now choose $\alpha$ least such that $h^C_\gamma <^* f_\alpha$ for all $C \in {\mathcal C}_\gamma$, which is
  possible since $\vert {\mathcal C}_\gamma \vert \le \kappa$,
  and then let $g_\gamma = f_\alpha$.

  We need to verify that points of uncountable cofinality are better, so let $\gamma < \kappa^+$ have uncountable cofinality
  and let $C \in {\mathcal C}_\gamma$. For every $\beta \in \lim(C)$ we have $C \cap \beta \in {\mathcal C}_\beta$,
  and there is $m$ so large that $\ot(C \cap \beta) < \kappa_m$ and
  $g_\beta(n) > h^{C\cap\beta}_\beta(n)$ for all $n \ge m$. 
  By definition $h^{C\cap\beta}_\beta(n) = \sup (\{ g_\alpha(n) : \alpha \in C \cap \beta \} )$ for $n \ge m$,
  so that $g_\alpha(n) > g_\beta(n)$ for all $\alpha \in C \cap \beta$ and all $n \ge m$.

\section{Disjointing scales and incompact lines}  

The following definition is implicit in \cite{CFM} but is worth making explicit here. Again
the definition was made for scales but works equally well for weak scales. 
\begin{definition}
  Let $(f_\alpha)_{\alpha < \kappa^+}$ be a weak scale of length $\kappa^+$ in $\prod_n \lambda_n$.
  Then $(f_\alpha)_{\alpha < \kappa^+}$ is {\em disjointing} if for every $\beta < \kappa^+$ there
  is a sequence of natural numbers $(n_\alpha: \alpha < \beta)$ such that
  for $f_\alpha(n) < f_{\alpha'}(n)$ for all 
  $\alpha < \alpha' < \beta$ and all $n \ge n_\alpha, n_{\alpha'}$.
\end{definition}

We note that a disjointing scale is another example of an incompact structure
of size $\kappa^+$. By the pigonhole principle it is never possible to choose
$(n_\alpha)_{\alpha < \kappa^+}$ with the property above.

It is proved in \cite{CFM} that a scale of length $\kappa^+$ in which every point of uncountable cofinality is
better is a disjointing scale, and the proof works equally well for weak scales.
The proof is quite short, so we sketch it here.

We proceed by induction. 
At points $\alpha$ of
countable cofinality we choose $(\alpha_n)_{n < \omega}$ increasing and cofinal in $\alpha$ with $\alpha_0 = 0$,
and choose values $(n_\beta)_{\alpha_n \le \beta < \alpha_{n+1}}$ which work in the interval $[\alpha_n, \alpha_{n+1})$,
  arranging in addition that $f_{\alpha_m}(t) < f_{\alpha_n}(t) \le f_\beta(t) < f_{\alpha_{n+1}}(t)$ for all
 $\beta \in [\alpha_n, \alpha_{n+1})$,  $m < n$ and $t \ge n_\beta$. 
    At points $\alpha$ of uncountable cofinality we fix a club subset $C$ of $\alpha$ with $\min(C) = 0$ witnessing that
    $\alpha$ is better, enumerate $\lim(C)$ in increasing order as $(\alpha_i)_{i < \cf(\alpha)}$, and  
   choose values $(n_\beta)_{\alpha_i \le \beta < \alpha_{i+1}}$ which work in the interval $[\alpha_i, \alpha_{i+1})$,
     arranging in addition that $f_{\alpha_j}(t) < f_{\alpha_i}(t) \le f_\beta(t) < f_{\alpha_{i+1}}(t)$ for all
     $\beta \in [\alpha_i, \alpha_{i+1})$,  $j < i$ and $t \ge n_\beta$.

For our purposes the main point of disjointing weak scales is that they give
incompact linear orderings of the type we discussed in the introduction.
The following Lemma distills the main point of Todorcevic's construction from
\cite[Section 7.6]{Stevobook}.
\begin{lemma} 
  Let $(f_\alpha)$ be a disjointing weak scale of length $\kappa^+$ in $\prod_n \lambda_n$,
  and let $L = \{ f_\alpha: \alpha < \kappa^+ \}$ ordered with the lexicographic ordering
  $\lex$. Then $L$ has a weakly dense subset of size $\kappa$, $L$ is
  not $\sigma$-scattered, and every small subordering of $L$ is $\sigma$-wellordered.
\end{lemma}
  
\begin{proof} The argument for the weakly dense subset is standard but we give it for the sake of completeness.
  For each $s \in \bigcup_{n < \omega} \prod_{i < n} \kappa_i$, let $g_s$ be some $f_\alpha$ with
  $f_\alpha \restriction \lh(s) = s$ if such a function exists,
  choosing $g_s$ as the $\lex$-minimal such $f_\alpha$ when there is a minimal one.

  Now let $f_\eta \lex f_\zeta$, let $t$ be the longest common initial segment
  of $f_\eta$ and $f_\zeta$, then let $n = \lh(t)$ and $u = f_\zeta \restriction n+1$.
  If $f_\zeta$ is $\lex$-minimal with initial segment $u$ then $f_\zeta = g_u$
  and we are done, otherwise there is $\mu$ such that $f_\mu \lex f_\zeta$
  and $f_\mu \restriction n + 1 = u$. Let $v$ be the longest common initial segment
  of $f_\mu$ and $f_\zeta$, then let $n' = \lh(v)$ and $w = f_\mu \restriction n'$.
  It is easy to see that $g_w$ exists, and $f_\eta \lex g_w \lex f_\zeta$.

  By facts from Section \ref{LOs}, $L$ does not embed $\kappa^+$ or its reverse and
  so $L$ is not $\sigma$-scattered. For the second claim, it will suffice to show
  that $\{ f_\alpha: \alpha < \beta \}$ is $\sigma$-wellordered by $\lex$
  for all $\beta < \kappa^+$. To see this let $(n_\alpha)_{\alpha < \beta}$ witness
  the disjointing property, and let $L_n = \{ f_\alpha: \mbox{$\alpha < \beta$ and $n_\alpha = n$} \}$.
  If $\alpha, \alpha' < \beta$ with $f_\alpha \lex f_{\alpha'}$ and $n_\alpha = n_\alpha'= n$, then 
  the first point of disagreement between $f_\alpha$ and $f_{\alpha'}$ is at most $n$,
  so that $f_\alpha \restriction n+1 \lex f_{\alpha'} \restriction n+1$: so the map
  $f \mapsto f \restriction n+1$ is a $\lex$-order preserving map from $L_n$
  to $\prod_{i \le n} \lambda_i$, which is a $\lex$-wellordered set.
\end{proof}

Combining the results so far, we have proved:
\begin{theorem} If $\kappa$ is a singular
cardinal of cofinality $\omega$ and the weak square principle $\square^*_\kappa$ holds, then there is a
linear ordering of cardinality $\kappa^+$ which is not $\sigma$-scattered, while every subordering of size at most $\kappa$
is $\sigma$-wellordered.
\end{theorem}

\section{Weak scales from failures of SCH}

In this section we show how to get  weak scales with many better points from certain failures of the
Singular Cardinals Hypothesis (SCH). We note that a failure of SCH can be viewed
as an instance of incompactness, which fits well into the general theme of this paper.
The arguments require a slightly heavier dose of PCF theory than in previous sections.
As before we will focus on parts of the theory which are most relevant
to our main goal. 
All the needed background on PCF theory
can be found in the excellent account by Abraham and Magidor \cite{AbMg}.

We start with the notion of an {\em exact upper bound (eub)}.
Let $\mu$ be a regular cardinal with $\mu > \kappa$ and let
$(f_\alpha)_{\alpha < \mu}$ be a scale in $\prod_n \kappa_n$,
where  $(\kappa_n)_{n < \omega}$ is an increasing sequence of regular
cardinals cofinal in $\kappa$. 
If $\beta$ is a limit ordinal with $\beta \le \mu$  then a function $h: \omega \rightarrow On$ is
an eub for $(f_\alpha)_{\alpha < \beta}$ if:
\begin{itemize}
\item $f_\alpha <^* h$ for all $\alpha < \beta$.
\item For all $g <^* h$ there is $\alpha < \beta$ such that  $g <^* f_\alpha$.
\end{itemize}

\begin{remark}
  When an eub exists for $(f_\alpha)_{\alpha < \beta}$ it is unique modulo finite.
  An eub for $(f_\alpha)_{\alpha < \beta}$ is a least upper bound (lub)
  for $(f_\alpha)_{\alpha < \beta}$
  in the natural sense, but in general being an eub is stronger than being an lub. 
  The function $n \mapsto \kappa_n$ is an eub for $(f_\alpha)_{\alpha < \mu}$.
\end{remark}

There is a close connection between the existence of eub's and the
concept of a good point:
\begin{enumerate}
\item The following are equivalent for a limit ordinal $\beta < \mu$ with
  $\omega < \beta < \kappa$:
  \begin{enumerate}
  \item $\beta$ is a good point.
  \item There exists an eub $h$ for $(f_\alpha)_{\alpha < \beta}$ such that
    $\cf(h(n) = \cf(\beta)$ for all $n$.
  \end{enumerate}
\item If $\beta < \mu$ with $\cf(\beta) = \kappa^+$, and
  the set of good points below $\beta$ of cofinallty $\eta$ is stationary 
  for unboundedly many regular cardinals $\eta < \kappa$, then
  there is an eub $h$ for $(f_\alpha)_{\alpha < \beta}$ such that
  $\cf(h(n)) \rightarrow \kappa$.
\end{enumerate}
The first of these facts is straightforward. The existence of a suitable eub
follows from Shelah's Trichotomy Theorem \cite[Exercise 2.27]{AbMg}.

%


Now we sketch a proof that if $\kappa$ is a singular strong limit cardinal
of cofinality $\omega$ and $2^\kappa > \kappa^+$, then there is a scale 
$(g_\alpha)_{\alpha < \kappa^{++}}$ in $\prod_n \kappa_n$ 
for some sequence of regular cardinals $(\kappa_n)_{n < \omega}$ which is
increasing and cofinal in $\kappa$. This fact, which was communicated to us by
Todd Eisworth, follows by combining various results by Shelah.

Combining \cite[Conclusion 5.10 (2)]{CardArith} and \cite[$\otimes_1$ page 40]{CardArith},
there are a countable cofinal set of regular cardinals $A \subseteq \kappa$ and a uniform ultrafilter $U$
on $A$ such that $\kappa^{++} = \cf(\prod A/U)$.  Now we proceed as follows, referrring the reader
to \cite{AbMg} for the necessary facts about true cofinalities, the ideals $J_{<\lambda}[A]$, and PCF generators: 
\begin{itemize}
\item Since $A$ is progressive and $\kappa^{++} \in \pcf(A)$, 
 there is a set $B \subseteq A$ which generates $J_{<\kappa^{+++}}[A]$ over $J_{<\kappa^{++}}[A]$,
in particular $\kappa^{++} = \max \pcf(B)$ and 
 $\tcf(\prod B/J_{<\kappa^{++}}[B]) = \kappa^{++}$.   
\item If $\kappa^+ \in \pcf(B)$ then there is a set $C \subseteq B$ which generates 
$J_{<\kappa^{++}}[B]$ over $J_{<\kappa^+}[B]$, and replacing $B$ by $B \setminus C$ we may
assume in addition that $\kappa^+ \notin \pcf(B)$ and hence $J_{<\kappa^{++}}[B] = J_{<\kappa^+}[B]$. 
\item Since $\kappa$ is singular $J_{<\kappa^+}[B] = J_{<\kappa}[B]$, and is easily seen to be contained
in the ideal $J_{\rm bd}[B]$ of bounded subsets of $B$. 
\item It follows that $\tcf(\prod B/J_{\rm bd}[B]) = \kappa^{++}$. Choosing a cofinal
  subset  of order type $\omega$ in $B$, we obtain $(\kappa_n)_{n < \omega}$ and $(g_\alpha)_{\alpha < \kappa^{++}}$ as required.
\end{itemize}

Given a scale $(g_\alpha)_{\alpha < \kappa^{++}}$ in $\prod_n \kappa_n$,
we will construct a weak scale of length $\kappa^+$ in $\prod_n \mu_n$ for
some  sequence $(\mu_n)_{n < \omega}$ of ordinals less than $\kappa$,
so that $\cf(\mu_n) \rightarrow \kappa$ and every point of uncountable cofinality
is better.

We start by fixing a sequence $({\mathcal C}_\alpha)_{\alpha < \kappa^+}$
such that:
  \begin{itemize}
  \item ${\mathcal C}_\alpha$ is a family of closed unbounded subsets of $\alpha$, each of order type
    less than $\kappa$, with $1 \le \vert {\mathcal C}_\alpha \vert \le \kappa^+$.
  \item For all $C \in {\mathcal C}_\alpha$ and all $\beta \in \lim(C)$,
    $C \cap \beta \in {\mathcal C}_\beta$.
  \end{itemize}  
  Such sequences, sometimes known as {\em silly squares}, are easy to construct: fix for every
  limit ordinal $\alpha$ less than $\kappa^+$ a club set $C_\alpha$ with $\ot(C_\alpha) < \kappa$,
  and let ${\mathcal C}_\beta = \{ C_\alpha \cap \beta: \beta \in \lim(C_\alpha) \cup \{ \alpha \} \}$.

  We will define a subsequence $(h_\beta)_{\beta < \kappa^+}$ of $(g_\alpha)_{\alpha < \kappa^{++}}$,
  such that all points of uncountable cofinality are better in $(h_\beta)_{\beta < \kappa^+}$.
  The construction is quite similar to
  the one we gave earlier building a better scale from $\square^*_\kappa$, but this
  time we bound a set of size at most $\kappa^+$ pointwise suprema, using the fact
  that we are starting with a scale of length $\kappa^{++}$. 
  
  Let $h_0 = g_0$, and let $h_{\beta + 1} = g_\alpha$ for $\alpha$ least such that $h_\beta <^* g_\alpha$.  
    Suppose that we defined $(h_\beta)_{\beta < \gamma}$ for some limit ordinal $\gamma$ less than $\kappa^+$. 
  For each $C \in {\mathcal C}_\gamma$
  let $m$ be such that $\ot(C) < \kappa_m$, and define
  $k^C_\gamma(n) = \sup (\{ g_\beta(n) : \beta \in C \} )$ for $n \ge m$.
  Now choose $h_\gamma = g_\alpha$ where $\alpha$ is least such that $k^C_\gamma <^* g_\alpha$
  for all $C \in {\mathcal C}_\gamma$. The argument that all points of uncountable cofinality
  are better is exactly as in the construction of a better scale from $\square^*_\kappa$.

 The sequence $(h_\beta)_{\beta < \kappa^+}$ is of course bounded modulo finite in $\prod_n \kappa_n$. 
 By the preceding discussion of exact upper bounds, since every point of uncountable cofinality
 is good we may choose an exact upper bound $n \mapsto \mu_n$ where $(\mu_n)_{n < \omega}$
 is a sequence of ordinals with $\mu_n < \kappa_n$ and $\cf(\mu_n) \rightarrow \kappa$.
 Then  $(h_\beta)_{\beta < \kappa^+}$ is a weak scale in $\prod_n \mu_n$, and every point of uncountable
 cofinality is better.

  We have proved: 
\begin{theorem}
 If $\kappa$ is singular strong limit of cofinality $\omega$ and $2^\kappa > \kappa^+$, then
    then there is a
linear ordering of cardinality $\kappa^+$ which has density $\kappa$ while every subordering of size at most $\kappa$
is $\sigma$-wellordered.
\end{theorem}

 \section*{Acknowledgements}  Many thanks to Todd Eisworth for
 explaining how to get a scale of length $\kappa^+$ from a failure of SCH at $\kappa$. Thanks also to Justin Moore for alerting me to
 Todorcevic's results in \cite{Stevobook}, and their connection with PCF, during a very enjoyable visit to Cornell.

\section*{Appendix: Minimal walks, $\rho$ functions and scales}

  In this section we sketch some constructions by Todorcevic \cite[Chapter 7]{Stevobook}
  and their connections with PCF-theoretic scales

  Let $\kappa$ be any singular cardinal,  
  and let $(C_\alpha)_{\alpha < \kappa^+}$ is a $\square_\kappa$ sequence
  where $\ot(C_\alpha) < \kappa$ for all $\alpha$. 
 $\Lambda(\alpha, \beta)$ is the largest limit point of $C_\beta \cap (\alpha+1)$,
    or zero if no such limit point exists.

  Define $\rho:[\kappa^+]^2 \rightarrow \kappa$ as follows: 
  By convention $\rho(\alpha, \beta)$ is zero for $\alpha = \beta$.
  For $\alpha < \beta$ it is the maximum of the ordinals:
\begin{itemize}
\item  $\ot(C_\beta \cap \alpha)$
\item  $\rho(\alpha, \min(C_\beta \setminus \alpha))$
\item The ordinals $\rho(\xi, \alpha)$ for
   $\xi \in C_\beta \cap [\Lambda(\alpha, \beta), \alpha)$
\end{itemize}

\begin{remark} The set of $\xi$ in the third clause is finite.
   $\alpha$ itself can be a limit point of $C_\beta$, but then
    $\Lambda(\alpha, \beta) = \alpha$ and the set of $\xi$ is empty.
In this case $\min(C_\beta \setminus \alpha) = \alpha$,
so that $\rho(\alpha, \beta) = \ot(C_\beta \cap \alpha)$.
It is also easy to see that if $\alpha \in \lim(C_\beta)$ then
$\rho(\zeta, \alpha) = \rho(\zeta, \beta)$ for all $\zeta < \alpha$. 
\end{remark}

We list the key properties of $\rho$, labelled so that we can
 trace which ones are used when we build scales:
\begin{enumerate}
\item \label{rho1}
  For all ordinals $\nu < \kappa$ and $\beta < \kappa^+$,
  $\vert \{ \alpha < \beta: \rho(\alpha, \beta) \le \nu \} \vert < \nu^+$
\item \label{rho2} For $\alpha < \beta < \gamma < \kappa^+$:   
\begin{enumerate}
\item \label{rho2a} $\rho(\alpha, \beta) \le \max(\rho(\alpha, \gamma), \rho(\beta, \gamma))$.
\item \label{rho2b} $\rho(\alpha, \gamma) \le \max(\rho(\alpha, \beta), \rho(\beta, \gamma))$.
\end{enumerate}
\item \label{rho3} 
  For all ordinals $\nu < \kappa$ and $\beta < \kappa^+$,
 the set $\{ \alpha < \beta: \rho(\alpha, \beta) \le \nu \}$ is closed in $\beta$. 
\end{enumerate}

\begin{remark} 
  Property \ref{rho1} is typical for $\rho$ functions defined from minimal walks,
  the proof requires only that $\rho(\alpha, \beta) \ge \max(\ot(C_\beta \cap \alpha), \rho(\alpha, \min(C_\beta \setminus \alpha)))$
and  $C_\alpha$ is club in $\alpha$. 
\end{remark}

\begin{remark} Readers of \cite[Chapter 7]{Stevobook} will note that the $\rho$ functions defined there
  are more general, and we are just defining $\rho$ in a special case.   
  It is rather simple to
  prove the properties of the $\rho$ function we conside here directly, by imitating
  the proofs for the analogous function $\rho: [\omega_1]^2 \rightarrow \omega$ from \cite[Chapter 3]{Stevobook} and using
  the coherence property inherited from the square. 
\end{remark}

 Assume now that $\cf(\kappa) = \omega$ and that $(\kappa_n)_{n < \omega}$ is an increasing
 sequence cofinal in $\kappa$. Note that the sets  $P_{\kappa_n}(\beta)$
 increase with $n$ and their union is $\beta$. 
 Todorcevic defined a sequence of functions
 $(f_\beta)_{\beta < \kappa^+}$ in $\prod_{n < \omega} \kappa_n^+$ by
 $f_\beta(n) = \ot(P_{\kappa_n}(\beta))$.

 The properties of $\rho$ imply certain properties for the functions $f_\beta$.
 We will keep track of exactly how the properties of $\rho$ are used.  
 Let $R_1$ abbreviate ``$\rho$ has property \ref{rho1}''
 $R_{12}$ abbreviate ``$\rho$ has properties \ref{rho1} and \ref{rho2}''
 and $R_{123}$ abbreviate ``$\rho$ has properties \ref{rho1}, \ref{rho2} and \ref{rho3}''

\begin{enumerate} 

\item \label{prop1}  $(R_1)$ For all $\beta < \kappa$, $f_\beta(n) < \kappa_n^+$. 

  This follows immediately from property $\ref{rho1}$ of $\rho$.

\item \label{prop2}  ($R_{12}$)  $f_\beta <^* f_\gamma$ for $\beta < \gamma < \kappa^+$.

  Let $\beta \in P_{\kappa_m}(\gamma)$, then for $n \ge m$ we have
  $P_{\kappa_n}(\beta) \cup \{ \beta \} \subseteq P_{\kappa_n}(\gamma)$
  by property \ref{rho2b} of $\rho$, 
 so that $f_\beta(n)  < f_\gamma(n)$.

\item  \label{prop3} ($R_{12}$) Let $\alpha < \beta < \gamma$ with $\alpha, \beta \in P_{\kappa_m}(\gamma)$.
 then  $f_\alpha(n) < f_\beta(n)$ for $n \ge m$.

 Observe that $\alpha \in P_{\kappa_m}(\beta)$  by property \ref{rho2a} of $\rho$,
 and argue as in the preceding item.

\item \label{prop4} ($R_{12}$) Let $\gamma < \kappa^+$ with $\cf(\gamma) > \omega$. Then
   $\gamma$ is a good point. 

   Let $m$ be such that $P_{\kappa_m}(\gamma)$ is unbounded in $\gamma$.
and use the preceding item. 
   
\item \label{prop5}
   ($R_{123}$) Let $\gamma < \kappa^+$ with $\cf(\gamma) > \omega$. Then
   $\gamma$ is a very good point. 

   The argument is the  same as in the preceding item, but now
   property \ref{rho3} gives that  $P_{\kappa_m}(\gamma)$ is club in $\gamma$.

 \end{enumerate}

\bibliographystyle{amsplain}
\bibliography{ijm_for_mm}

\end{document}